\documentclass[letter]{amsart}
\usepackage{accents}
\usepackage{graphicx}
\usepackage{amsfonts}
\usepackage{amscd}
\usepackage{amssymb}
\usepackage{mathtools}
\usepackage{booktabs}
\usepackage{hyperref}
\usepackage{url}
\usepackage{tikz}
\usetikzlibrary{decorations.pathmorphing}
\usetikzlibrary{decorations.markings}
\usetikzlibrary{arrows.meta}

\tikzset{middlearrow/.style={
        decoration={markings,
            mark= at position 0.5 with {\arrow{Stealth[red]}} ,
        },
        postaction={decorate}
    }
}

\usepackage{amsthm}
\usepackage{verbatim}
\usepackage{calligra}
\usetikzlibrary{shapes}
\usepackage{multirow}
\usepackage[savepos]{zref}
\newtheorem{thm}{Theorem}[section]
\newtheorem{corollary}[thm]{Corollary}

\theoremstyle{definition}

\theoremstyle{remark}

\numberwithin{thm}{subsection}
\numberwithin{equation}{subsection}

\usepackage[mathscr]{euscript}
\usepackage[T2A,T1]{fontenc}

  {\begin{description}%
    \setlength{\itemsep}{2.5pt}%
    \setlength{\parskip}{5pt}}%
  {\end{description}}

\DeclareMathOperator{\GCD}{GCD}

\newcommand{\From}{\colon}
\newcommand{\inar}{\ar@{^{(}->}}
\newcommand{\onar}{\ar@{->>}}

\usepackage{accents}
\newlength{\dtildeheight}

\newcommand{\Matrix}[4]{ \left( \begin{array}{cc}  #1 & #2 \\  #3 & #4 \\ \end{array} \right) }

\makeatletter
\newcommand{\raisemath}[1]{\mathpalette{\raisem@th{#1}}}
\newcommand{\raisem@th}[3]{\raisebox{#1}{$#2#3$}}
\makeatother

\newcommand{\One}{\mathbf{1}}

\newcommand{\abs}[1]{\left\vert#1\right\vert}

\newcommand{\alg}[1]{\boldsymbol{\mathrm{#1}}}

\newcommand{\GG}{\mathbb G}
\newcommand{\EE}{\mathbb E}
\newcommand{\ZZ}{\mathbb Z}
\newcommand{\QQ}{\mathbb Q}

\newcommand{\ident}{\equiv}

\newcommand{\To}{\rightarrow}

\newcommand{\isom}{\cong}

\makeatletter
\newcommand\@biprod[1]{%
  \vcenter{\hbox{\ooalign{$#1\prod$\cr$#1\coprod$\cr}}}}
\newcommand\biprod{\mathop{\mathpalette\@biprod\relax}\displaylimits}
\makeatother

\DeclareMathAlphabet{\mathcalligra}{T1}{calligra}{m}{n}

\DeclareMathOperator{\red}{red}
\DeclareMathOperator{\blue}{blue}
\DeclareMathOperator{\Cl}{Cl}

\title{The arithmetic of arithmetic Coxeter groups}

\author{Suzana Milea, Christopher Shelley and Martin H. Weissman}

\address{Department of Mathematics, University of California, Santa Cruz, 1156 High Street, Santa Cruz, CA 95064}
\email{weissman@ucsc.edu}
\date{\today}

\begin{document}
\begin{abstract}
In the 1990s, J.H.~Conway published a combinatorial-geometric method for analyzing integer-valued binary quadratic forms (BQFs).  Using a visualization he named the ``topograph,'' Conway revisited the reduction of BQFs and the solution of quadratic Diophantine equations such as Pell's equation.  It appears that the crux of his method is the coincidence between the arithmetic group $PGL_2({\mathbb Z})$ and the Coxeter group of type $(3,\infty)$.  There are many arithmetic Coxeter groups, and each may have unforeseen applications to arithmetic.  We introduce Conway's topograph, and generalizations to other arithmetic Coxeter groups.  This includes a study of ``arithmetic flags'' and variants of binary quadratic forms.
\end{abstract}

\maketitle

\section{Conway's topograph}

Binary quadratic forms (BQFs) are functions $Q \From \ZZ^2 \To \ZZ$ of the form $Q(x,y) = ax^2 + bxy + cy^2$, for some integers $a,b,c$.  The discriminant of such a form is the integer $\Delta = b^2 - 4ac$.  In \cite{Conway}, J.H.~Conway visualized the values of a BQF through an invention he called the {\em topograph}.  

\subsection{The geometry of the topograph}
The topograph is an arrangement of points, edges, and faces, as described below.
\begin{itemize}
\item
Faces correspond to {\em primitive lax vectors}: coprime ordered pairs $\vec v = (x,y) \in \ZZ^2$, modulo the relation $(x,y) \sim (-x,-y)$.  Such a vector will be written $\pm \vec v$.
\item
Edges correspond to {\em lax bases}:  unordered pairs $\{ \pm \vec v, \pm \vec w \}$ of primitive lax vectors which form a $\ZZ$-basis of $\ZZ^2$.  (Clearly this is independent of sign choices.)
\item
Points correspond to {\em lax superbases}: unordered triples $\{ \pm \vec u, \pm \vec v, \pm \vec w \}$, any two of which form a lax basis.
\end{itemize}

Incidence among points, edges, and faces is defined by containment.  A {\em maximal arithmetic flag} in this context refers to a point contained in an edge contained in a face.  The geometry is displayed in Figure \ref{ConwayDT}; the points and edges form a ternary regular tree, and the faces are $\infty$-gons.
\begin{figure}
\centering
\includegraphics[width=.8\linewidth]{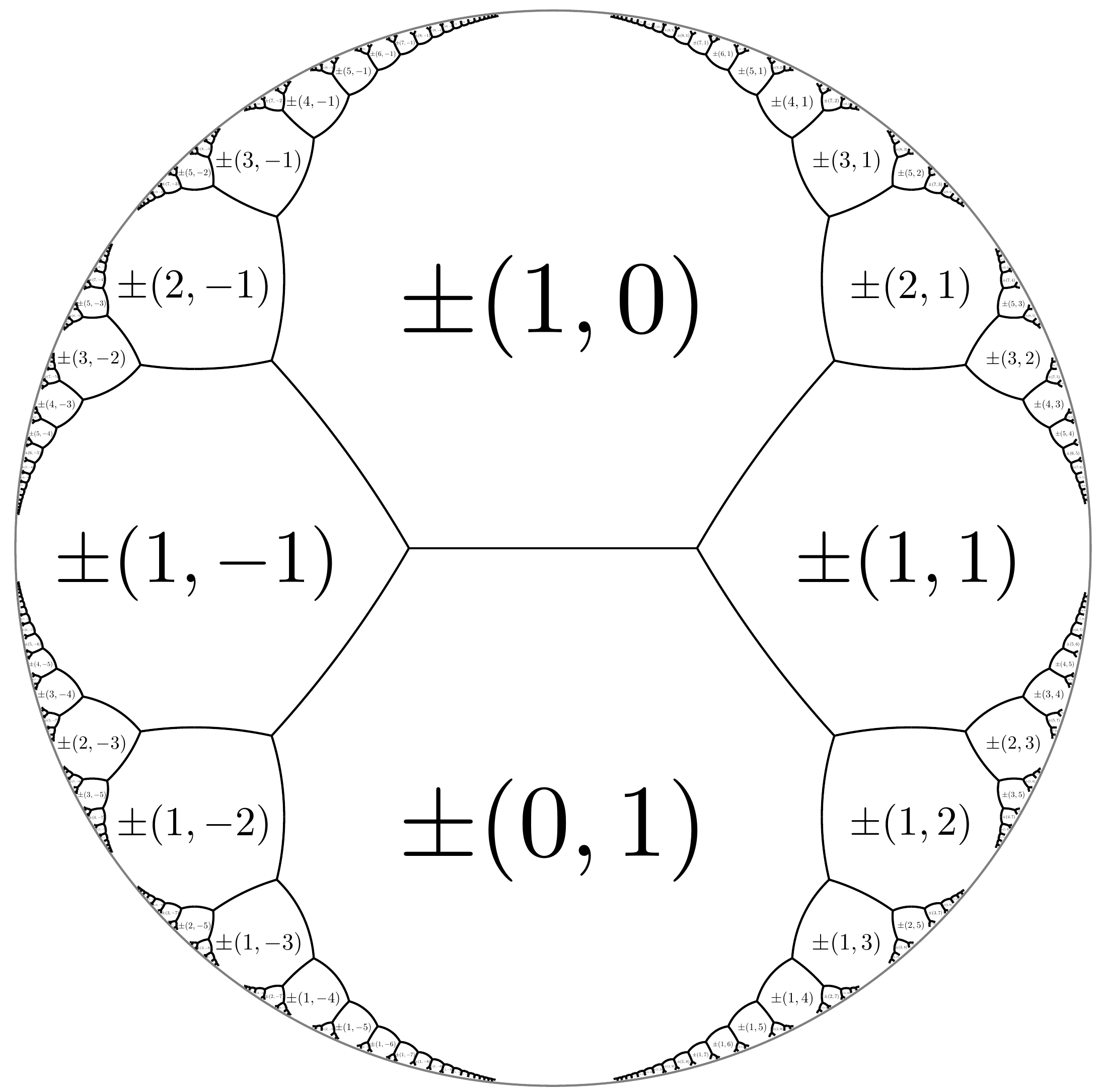}
\caption{Conway's geometry of primitive lax vectors, lax bases, and lax superbases.}
\label{ConwayDT}
\end{figure}
The group $PGL_2(\ZZ) = GL_2(\ZZ) / \{ \pm \One \}$ acts on the set of primitive lax vectors, and likewise on lax bases and superbases.  The action is simply-transitive on maximal arithmetic flags.

On the other hand, the geometry of Figure \ref{ConwayDT} also arises as the geometry of flags in the {\em Coxeter group} of type $(3,\infty)$.  This is the Coxeter group with diagram 
\tikz{
\filldraw (0,0) circle (0.05); 
\filldraw (1,0) circle (0.05); 
\filldraw (2,0) circle (0.05);
\draw (0,0) to node[above] {$3$} (1,0);
\draw (1,0) to node[above] {$\infty$} (2,0); 
}.  The group $W$ encoded by such a diagram is generated by elements $S = \{ s_0, s_1, s_2 \}$ corresponding to the nodes, modulo the relations $s_i^2 = 1$ (for $i = 0,1,2$), $s_0 s_2 = s_2 s_0$, and $(s_0 s_1)^3 = 1$.  If $T \subset S$ is a subset of nodes, write $W_T$ for the subgroup generated by $T$; it is called a {\em parabolic} subgroup.  The {\em flags} of type $T$ are the cosets $W / W_T$.  Incidence of flags is defined by intersection of cosets.  The Coxeter group $W$ acts simply-transitively on the maximal flags, i.e., the cosets $W / W_\emptyset = W$.

The geometric coincidence reflects the fact that $PGL_2(\ZZ)$ is isomorphic to the Coxeter group $W$ of type $(3,\infty)$, a classical result known to Poincar\'e and Klein.  But Conway's study of lax vectors, bases, and superbases goes further, giving an arithmetic interpretation of the flags for the Coxeter group.  This raises the natural question:  given a coincidence between an arithmetic group and a Coxeter group, is there an arithmetic interpretation of the flags in the Coxeter group?  We suggest a positive and interesting answer in later sections.

\subsection{Binary quadratic forms}

If one draws the values $Q(\pm \vec v)$ on the faces labeled by the primitive lax vectors $\pm \vec v$, one obtains Conway's {\em topograph} of $Q$.  Figures  \ref{TopGrow} and \ref{TopRiver} display examples.  If $u,v,e,f$ appear on the topograph of $Q$, in a local arrangement we call a {\em cell}, then Conway observes that the integers $e, u+v, f$ form an arithmetic progression.
\begin{center}
\begin{tikzpicture}[scale=0.5]
\draw (0,0) -- (1,0);
\draw (1,0) -- (1.5, 0.866);
\draw (1,0) -- (1.5, -0.866);
\draw (0,0) -- (120:1);
\draw (0,0) -- (240:1);
\draw (0.5, 0) node[above] {$u$} node[below] {$v$};
\draw (0,0) node[left] {$e$};
\draw (1,0) node[right] {$f$};
\draw (3,0) node[right] {$f - (u+v) = (u+v) - e.$ };
\end{tikzpicture}
\end{center}

The discriminant of $Q$ can be seen locally in the topograph, at every cell, by the formula  $\Delta = u^2 + v^2 + e^2 - 2uv - 2ve - 2eu = (u-v)^2 - ef$.


\begin{figure}[tbhp!]
\centering
\includegraphics[width=.8\linewidth]{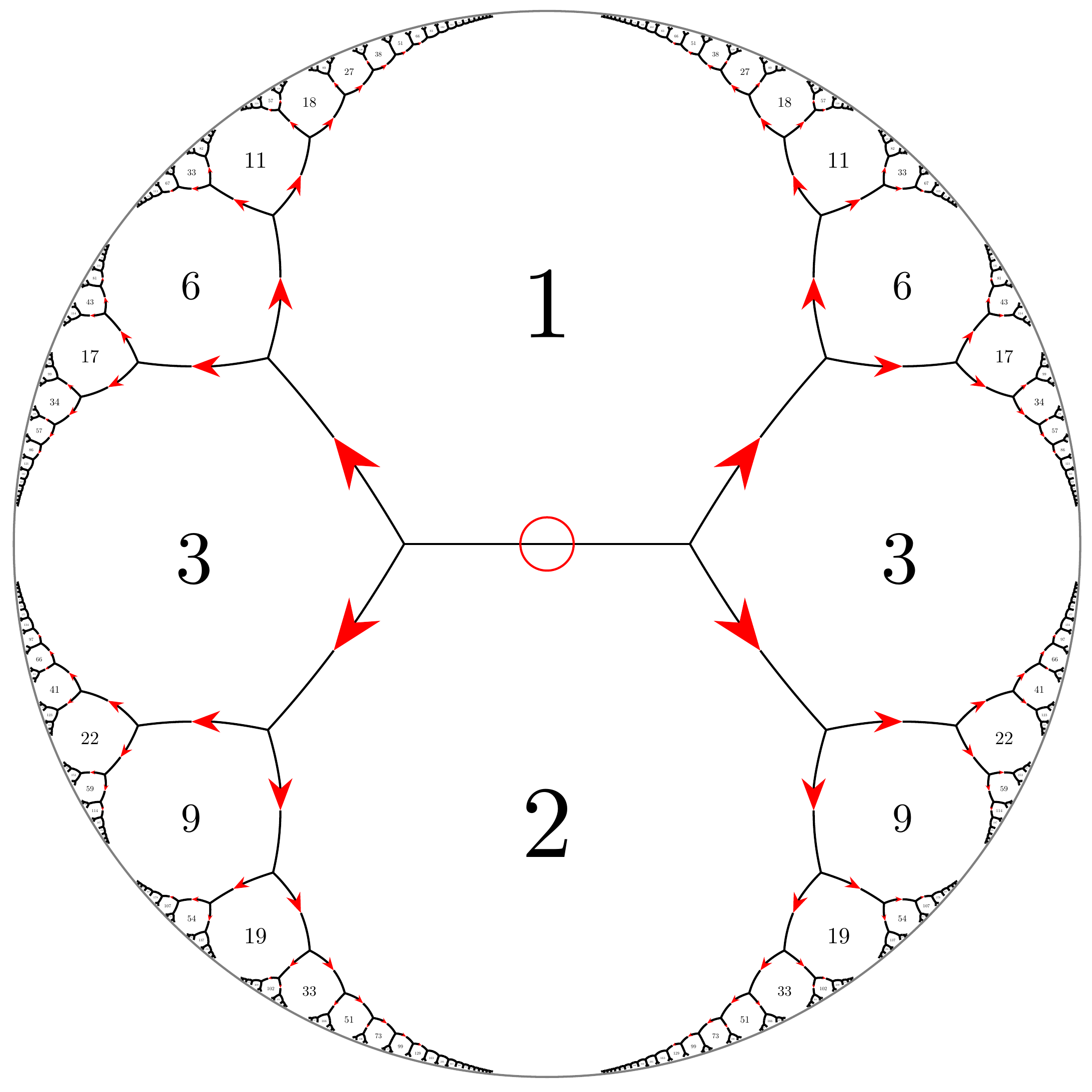}
\caption{The topograph of $Q(x,y) = x^2 + 2y^2$, with arrows exhibiting the climbing principle. The well (source of the flow) is the cell at the center of the figure.}
\label{TopGrow}
\end{figure}

A consequence of the arithmetic progression property is Conway's climbing principle; if all values in a cell are positive, place arrows along the edges in the directions of increasing arithmetic progressions.  Then every arrow propagates into two arrows; the resulting {\em flow} along the edges can have a source, but never a sink.  This implies the existence and uniqueness of a {\em well} for positive-definite forms: a triad or cell which is the source for the flow.  

The well gives the unique Gauss-reduced form $Q_{\mathrm{Gr}}$ in the $SL_2(\ZZ)$-equivalence class of $Q$.  More precisely, every well contains a triple $u \leq v \leq w$ of positive integers satisfying $u + v \geq w$, with strict inequality at triad-wells and equality at cell-wells (see Figure \ref{TopGrow}).  Depending on the orientation of $u,v,w$ at the well, the Gauss-reduced form is given in Figure \ref{wellreduced}; in the ambiguously-oriented case with $u = v$, $Q_{\mathrm{Gr}}(x,y) = u x^2 + (u+v-w) xy + vy^2$.  If $u+v = w$, both orientations occur in a cell-well, and $Q_{\mathrm{Gr}}(x,y) = ux^2 + v y^2$.  
\begin{figure}
\begin{center}
\begin{tikzpicture}[scale=0.6]
\draw[middlearrow] (0,0) -- (1,0);
\draw[middlearrow] (0,0) -- (120:1);
\draw[middlearrow] (0,0) -- (240:1);
\draw[red, thick] (0.65,0) circle (0.1);
\draw (0.5, 0) node[above] {$u$} node[below] {$v$};
\draw (0,0) node[left] {$w$};
\draw[|->] (1.5,0) -- (2.5,0) node[right] {$Q_{\mathrm{Gr}}(x,y) = ux^2 + (u+v-w) xy + vy^2$};
\end{tikzpicture}
\medskip

\begin{tikzpicture}[scale=0.6]
\draw[middlearrow] (0,0) -- (-1,0);
\draw[middlearrow] (0,0) -- (60:1);
\draw[middlearrow] (0,0) -- (-60:1);
\draw[red, thick] (-0.65,0) circle (0.1);
\draw (-0.5, 0) node[above] {$u$} node[below] {$v$};
\draw (-0,0) node[right] {$w$};
\draw[|->] (1.5,0) -- (2.5,0) node[right] {$Q_{\mathrm{Gr}}(x,y) = ux^2 - (u+v-w) xy + vy^2$};
\end{tikzpicture}
\end{center}
\caption{Every well corresponds to a Gauss-reduced form; in both diagrams, we assume $u \leq v \leq w$.  The Gauss-reduced form depends on the orientation.}
\label{wellreduced}
\end{figure}

When $Q$ is a nondegenerate indefinite form, Conway defines the {\em river} of $Q$ to be the set of edges which separate a positive value from a negative value in the topograph of $Q$.  Since all values on the topograph of $Q$ must be positive or negative, the river cannot branch or terminate.  The climbing principle implies uniqueness of the river.  Thus the river is a set of edges comprising a single endless line.  Bounding the values adjacent to the river implies periodicity of values adjacent to the river, and thus the infinitude of solutions to Pell's equation.  This is described in detail in \cite{Conway}.  {\em Riverbends} -- cells with a river as drawn below -- correspond to Gauss's reduced forms in the equivalence class of $Q$.
\begin{center}
\begin{tikzpicture}[scale=0.8,  decoration=snake]
\begin{scope}[xshift = -2.5cm]
\draw[cyan, thick, decorate] (0,0) -- (1,0);
\draw (1,0) -- (1.5, 0.866);
\draw[cyan, thick, decorate] (1,0) -- (1.5, -0.866);
\draw[cyan, thick, decorate] (0,0) -- (120:1);
\draw (0,0) -- (240:1);
\draw (0.5, 0) node[above=5] {$u > 0$} node[below=5] {$v < 0$};
\draw (0,0) node[left=5] {$e < 0$};
\draw (1,0) node[right=5] {$f > 0$};
\end{scope}
\begin{scope}[xshift = 2.5cm]
\draw[cyan, thick, decorate] (0,0) -- (1,0);
\draw[cyan, thick, decorate] (1,0) -- (1.5, 0.866);
\draw (1,0) -- (1.5, -0.866);
\draw (0,0) -- (120:1);
\draw[cyan, thick, decorate] (0,0) -- (240:1);
\draw (0.5, 0) node[above=5] {$u > 0$} node[below=5] {$v < 0$};
\draw (0,0) node[left=5] {$e > 0$};
\draw (1,0) node[right=5] {$f < 0$};
\end{scope}
\end{tikzpicture}
\end{center}
The existence of riverbends gives a classical bound, by an argument we learned from Gordan Savin.
\begin{thm}
If $Q$ is a nondegenerate indefinite BQF, then the minimum nonzero value $\mu_Q$ of $Q$ satisfies $\vert \mu_Q \vert \leq \sqrt{\Delta / 5}$.
\end{thm}
\begin{proof}
At a riverbend, one finds $\Delta = (u-v)^2 - ef = u^2 + v^2 - uv - vu - ef$, the sum of five {\em positive} integers.  It follows that one of $u^2, v^2, -uv, -vu, -ef$ must be bounded by $\Delta/5$.  Among $\vert u \vert, \vert v \vert, \vert e \vert, \vert f \vert$, one must be bounded by $\sqrt{\Delta/5}$.
\end{proof}

These are some highlights and applications of Conway's topograph.  In the next sections, we describe generalizations.
\begin{figure}
\centering
\includegraphics[width=.8\linewidth]{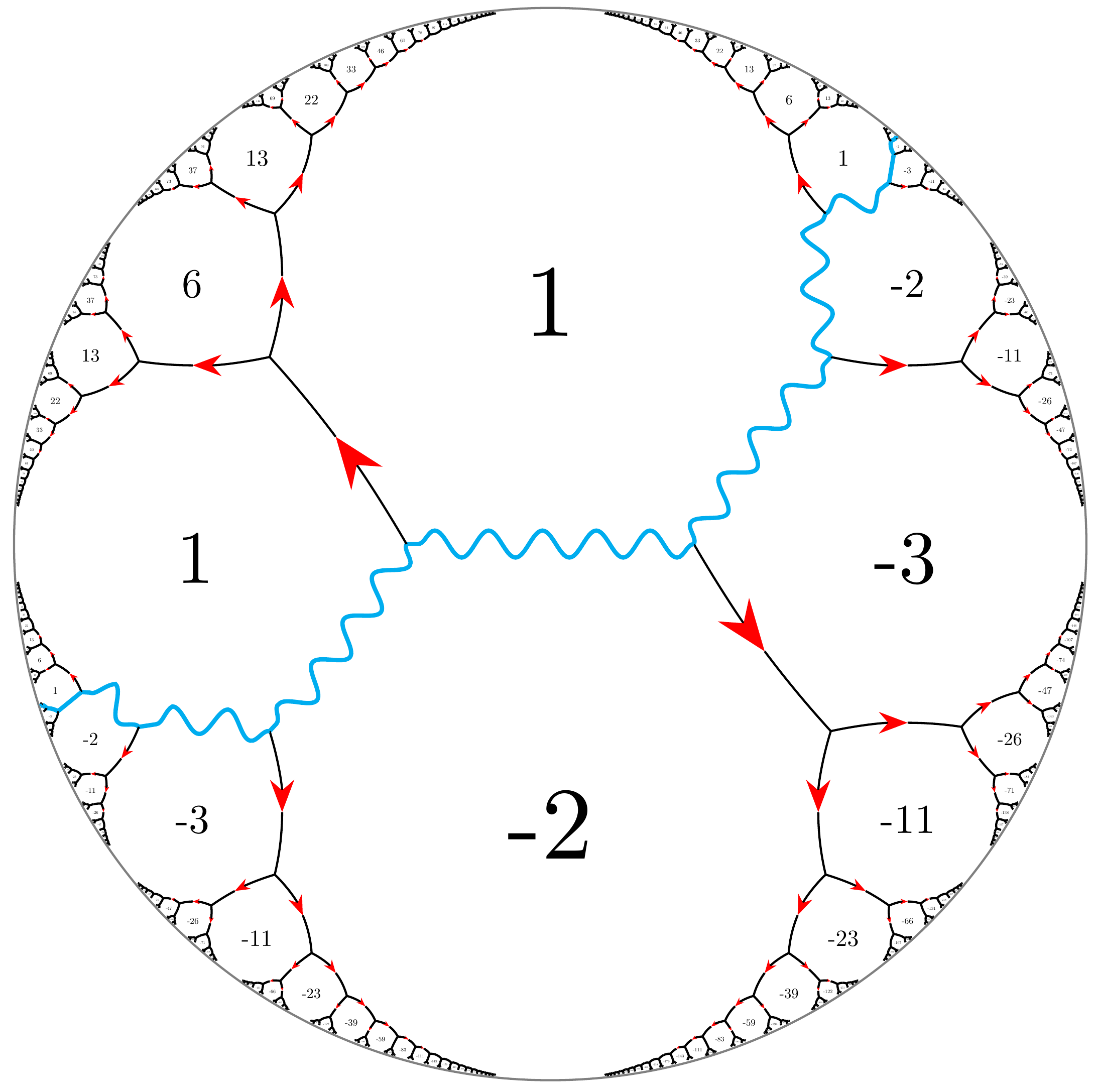}
\caption{The topograph of $Q(x,y) = x^2 - 3y^2$, exhibiting a periodic river.  Solutions to Pell's equation $x^2 - 3y^2 = 1$ are found along the riverbank.}
\label{TopRiver}
\end{figure}

\section{Gaussian and Eisenstein analogues}

Let $\GG$ denote the Gaussian integers:  $\GG = \ZZ[i]$.  Let $\EE$ denote the Eisenstein integers:  $\EE = \ZZ[e^{2 \pi i / 3}]$.  

\subsection{Arithmetic flags and honeycombs}

One may generalize Conway's vectors, bases, and superbases to arithmetic structures in $\GG^2$ and $\EE^2$. Guiding this are embeddings of $PSL_2(\GG)$ and $PSL_2(\EE)$ into hyperbolic Coxeter groups.  In \cite[\S1.I, 1.II, \S3, \S6]{Bianchi}, Bianchi describes generators for $SL_2(\GG)$ and $SL_2(\EE)$, and fundamental polyhedra for their action on hyperbolic 3-space.  Using reflections in the faces of these polyhedra, one may write explicit presentations of these groups; Fricke and Klein carry this out for $SL_2(\GG)$ in \cite[I.\S8]{FK}, where one finds a connection to the (later-named) Coxeter group of type $(3,4,4)$.  Schulte and Weiss give a detailed treatment, proving the following in \cite[Theorems 7.1,9.1]{SW}.
\begin{thm}
\label{GEgroups}
$PSL_2(\GG)$ is isomorphic to an index-two subgroup of $(3,4,4)^+$.  $PSL_2(\EE)$ is isomorphic to an index-two subgroup of $(3,3,6)^+$.   
\end{thm}
Here $(a,b,c)^+$ denotes the even subgroup of the Coxeter group of type $(a,b,c)$.  As the Coxeter groups of types $(3,4,4)$ and $(3,3,6)$ are commensurable to $PSL_2(\GG)$ and $PSL_2(\EE)$, respectively, we expect an arithmetic interpretation of the Coxeter geometries.  Such an arithmetic incidence geometry is described below.

\begin{itemize}
\item
Cells correspond to {\em primitive lax vectors}: coprime ordered pairs $\vec v = (x,y) \in \GG^2$ (respectively $\EE^2$), modulo the relation $(x,y) \sim (\epsilon x, \epsilon y)$ for all $\epsilon \in \GG^\times$ (resp., $\epsilon \in \EE^\times$).
\item
Faces correspond to {\em lax bases}: unordered pairs $\{ \epsilon \vec v, \epsilon \vec w \}$ of primitive lax vectors which form a $\GG$-basis of $\GG^2$ (respectively $\EE$-basis of $\EE^2$).
\item
Edges correspond to {\em lax superbases}:  unordered triples $\{ \epsilon \vec u, \epsilon \vec v, \epsilon \vec w \}$, any two of which form a lax basis.
\item
Points of the Eisenstein topograph correspond to {\em lax tetrabases}: unordered quadruples $\{ \epsilon \vec s,\epsilon\vec t, \epsilon\vec u, \epsilon\vec v \}$, any three of which form a lax superbasis.
\item
Points of the Gaussian topograph correspond to {\em lax cubases}: sets of three two-element sets $\{ \{ \epsilon \vec u_1, \epsilon \vec u_2 \}, \{ \epsilon \vec v_1, \epsilon \vec v_2 \},  \{ \epsilon \vec w_1, \epsilon \vec w_2 \} \}$, such that all eight choices of $i,j,k \in \{ 1,2 \}$ give a lax superbasis $\{ \epsilon \vec u_i, \epsilon \vec v_j, \epsilon \vec w_k \}$. 
\end{itemize}
Incidence is given by the obvious containments described above.  We call this incidence geometry the {\em topograph} for $\EE$ or $\GG$, and it is equipped with an action of $PSL_2(\EE)$ and $PSL_2(\GG)$, respectively.  The terms {\em tetrabasis} and {\em cubasis} reflect the residual geometry around a point (see Figure \ref{TetraCube}). Both geometries produce regular hyperbolic honeycombs \cite[Ch.~IV]{Coxeter}; the points, edges, and faces around each cell form square or hexagonal planar tilings in the Gaussian or Eisenstein case, respectively.  

The Gaussian and Eisenstein topographs are described by Bestvina and Savin in \cite[\S 7,8]{BS}.  Both topographs, and the following link to Coxeter geometries, are given in the PhD thesis of the second author.  
\begin{figure}
\centering
\includegraphics[width=.8\linewidth]{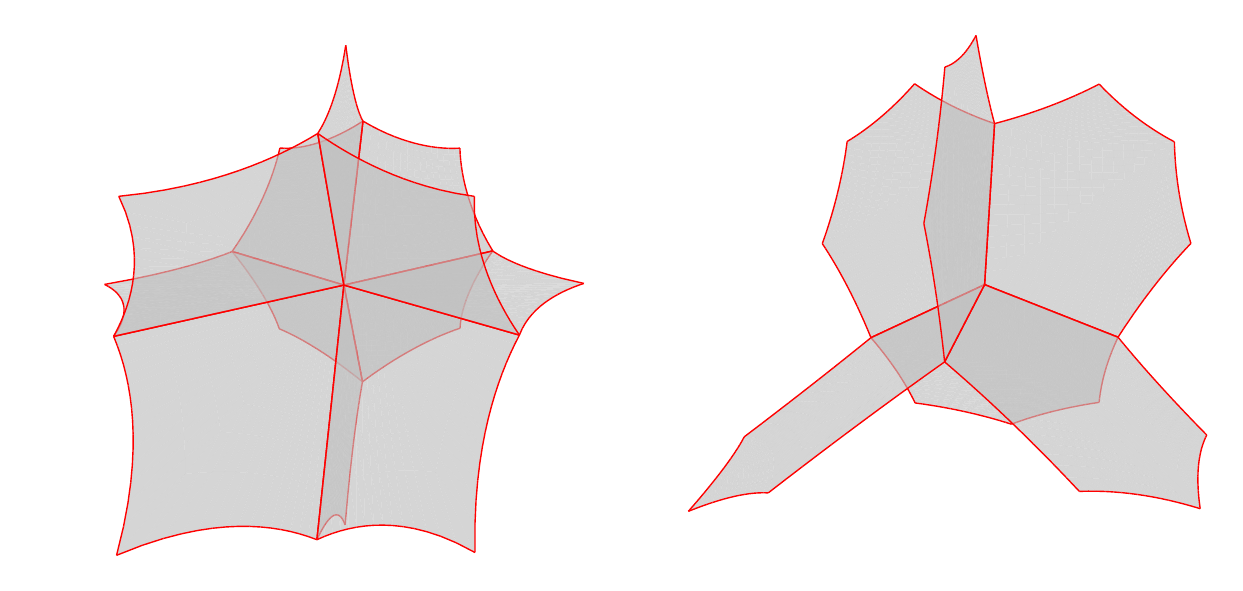}
\caption{The geometry of the Gaussian and Eisenstein topographs, displaying square and hexagonal faces, and cubic and tetrahedral residues at a point.}
\label{TetraCube}
\end{figure}

\begin{thm}
The topographs for $\EE^2$ and $\GG^2$ are equivariantly isomorphic to the Coxeter geometries of types (3,3,6) and (3,4,4), respectively.
\end{thm}
By equivariance, we mean that the isomorphism intertwines the natural actions of $PSL_2(\EE)$ and $PSL_2(\GG)$ on one hand with the actions of the Coxeter groups on the other, via the inclusion described in Theorem \ref{GEgroups}.


\subsection{Binary Hermitian forms}

An integer-valued binary Hermitian form (BHF), over $\EE$ or $\GG$, is a function $H \From \EE^2 \To \ZZ$ or $\GG^2 \To \ZZ$, of the form
$$H(x,y) = a x \bar x + \beta \bar x y + \bar \beta x \bar y + c y \bar y.$$
Here we assume $a,c \in \ZZ$, and $\beta \in (1 - \omega)^{-1} \EE$ or $\beta \in (1 + i)^{-1} \GG$ (the inverse different of $\EE$ or $\GG$, respectively).  The discriminant of $H$ is the integer defined by
$$\Delta = (3 \text{ or } 4) \cdot (\beta \bar \beta - ac), \text{ for $\EE$ or $\GG$, respectively.}$$

Fricke and Klein discuss reduction theory of Hermitian forms over $\GG$, using the geometry of $SL_2(\GG)$, in \cite[III.1, \S1--8]{FK}.  The topographs give a new approach, pursued by Bestvina and Savin \cite{BS}.  

Let $H$ be a BHF over $\EE$ or $\GG$.  Recalling that cells of the topographs correspond to primitive lax vectors, we define the {\em topograph} of $H$ to be the result of placing the value $H(\epsilon \vec v)$ at the topograph-cell marked by the primitive lax vector $\epsilon \vec v$.  The most interesting case occurs when $H$ is nondegenerate indefinite (taking positive and negative values, but never zero on a nonzero vector input).  In this case, Conway's river is replaced by the {\em ocean} -- the set of faces separating a cell with positive value from one with negative value.  Bestvina and Savin prove \cite[Theorems 5.3, 6.1]{BS} that this ocean is topologically an open disk, locally CAT(0) as a metric space, and the unitary group $U(H)$ acts cocompactly on the ocean.

Reduced indefinite BQFs correspond to riverbends in Conway's topograph.  In a similar way, one finds reduced indefinite BHFs at the points of the ocean of negative curvature, i.e., where more than four ocean-squares (for $\GG$) or more than three ocean-hexagons (for $\EE$) meet at a point.  From this, Bestvina and Savin \cite[Theorem 8.7]{BS} recover the optimal bound on the minima of nondegenerate indefinite BHFs over $\EE$.  The bound for $\GG$ can be obtained by the same method.
\begin{thm}
Let $H$ be a nondegenerate indefinite BHF.  Then the minimum nonzero value $\mu_H$ satisfies $\abs{ \mu_H }\leq \sqrt{\Delta/6}$.
\end{thm}
\begin{proof}
The Eisenstein case is proven in \cite{BS}, so we prove the Gaussian case.  Consider a vertex at which the ocean of $H$ has negative curvature; such a point exists by \cite[Corollary 6.2]{BS}.  The residue of the topograph at this vertex is a cube, whose faces are labeled by the values of $H$.  The intersection of the ocean with this cube forms a simple closed path on the edges, separating positive-valued faces from negative.
\begin{center}
\begin{tikzpicture}[scale=0.95,  decoration=snake]
\begin{scope}
\foreach \i in {0,1}
\foreach \j in {0,1}
\foreach \k in {0,1}
{
\pgfmathsetmacro{\x}{\i + \k*0.33}
\pgfmathsetmacro{\y}{\j + \k*0.67}
\coordinate (P\i\j\k) at (\x, \y);
}
\draw (P000) -- (P001);
\draw (P000) -- (P010);
\draw (P000) -- (P100);
\draw (P001) -- (P011);
\draw (P001) -- (P101);
\draw[cyan, thick] (P010) -- (P011);
\draw[cyan, thick] (P010) -- (P110);
\draw (P100) -- (P110);
\draw (P100) -- (P101);
\draw[cyan, thick] (P110) -- (P111);
\draw (P101) -- (P111);
\draw[cyan, thick] (P011) -- (P111);

\fill[black!80!green, opacity=0.3] (P010) -- (P110) -- (P111) -- (P011) -- cycle;

\draw (0.5, -0.3) node {(I)};
\end{scope}

\begin{scope}[xshift = 2.5cm]
\foreach \i in {0,1}
\foreach \j in {0,1}
\foreach \k in {0,1}
{
\pgfmathsetmacro{\x}{\i + \k*0.33}
\pgfmathsetmacro{\y}{\j + \k*0.67}
\coordinate (P\i\j\k) at (\x, \y);
}
\draw (P000) -- (P001);
\draw (P000) -- (P010);
\draw (P000) -- (P100);
\draw (P001) -- (P011);
\draw (P001) -- (P101);
\draw[cyan, thick] (P010) -- (P011);
\draw[cyan, thick] (P010) -- (P110);
\draw[cyan, thick] (P100) -- (P110);
\draw[cyan, thick] (P100) -- (P101);
\draw (P110) -- (P111);
\draw[cyan, thick] (P101) -- (P111);
\draw[cyan, thick] (P011) -- (P111);

\fill[black!80!green, opacity=0.3] (P010) -- (P110) -- (P111) -- (P011) -- cycle;
\fill[black!80!green, opacity=0.3] (P110) -- (P111) -- (P101) -- (P100) -- cycle;

\draw (0.5, -0.3) node {(II)};
\end{scope}

\begin{scope}[xshift = 5cm]
\foreach \i in {0,1}
\foreach \j in {0,1}
\foreach \k in {0,1}
{
\pgfmathsetmacro{\x}{\i + \k*0.33}
\pgfmathsetmacro{\y}{\j + \k*0.67}
\coordinate (P\i\j\k) at (\x, \y);
}
\draw (P000) -- (P001);
\draw (P000) -- (P010);
\draw (P000) -- (P100);
\draw[cyan, thick] (P001) -- (P011);
\draw[cyan, thick] (P001) -- (P101);
\draw[cyan, thick]  (P010) -- (P011);
\draw[cyan, thick]  (P010) -- (P110);
\draw[cyan, thick]  (P100) -- (P110);
\draw[cyan, thick]  (P100) -- (P101);
\draw (P110) -- (P111);
\draw (P101) -- (P111);
\draw (P011) -- (P111);

\fill[black!80!green, opacity=0.3] (P010) -- (P110) -- (P111) -- (P011) -- cycle;
\fill[black!80!green, opacity=0.3] (P110) -- (P111) -- (P101) -- (P100) -- cycle;
\fill[black!80!green, opacity=0.3] (P001) -- (P101) -- (P111) -- (P011) -- cycle;

\draw (0.5, -0.3) node {(III)};
\end{scope}

\begin{scope}[xshift = 7.5cm]
\foreach \i in {0,1}
\foreach \j in {0,1}
\foreach \k in {0,1}
{
\pgfmathsetmacro{\x}{\i + \k*0.33}
\pgfmathsetmacro{\y}{\j + \k*0.67}
\coordinate (P\i\j\k) at (\x, \y);
}
\draw[cyan, thick]  (P000) -- (P001);
\draw  (P000) -- (P010);
\draw[cyan, thick]  (P000) -- (P100);
\draw (P001) -- (P011);
\draw[cyan, thick]  (P001) -- (P101);
\draw[cyan, thick] (P010) -- (P011);
\draw[cyan, thick] (P010) -- (P110);
\draw[cyan, thick] (P100) -- (P110);
\draw(P100) -- (P101);
\draw (P110) -- (P111);
\draw[cyan, thick] (P101) -- (P111);
\draw[cyan, thick] (P011) -- (P111);

\fill[black!80!green, opacity=0.3] (P010) -- (P110) -- (P111) -- (P011) -- cycle;
\fill[black!80!green, opacity=0.3] (P110) -- (P111) -- (P101) -- (P100) -- cycle;
\fill[black!80!green, opacity=0.3] (P000) -- (P100) -- (P101) -- (P001) -- cycle;

\draw (0.5, -0.3) node {(IV)};
\end{scope}
\end{tikzpicture}
\end{center}
Form (I) corresponds to a Euclidean vertex; forms (II), (III), and (IV) correspond to ocean-vertices of negative curvature.  Label the values of $H$ on the cube by $a,b,c,u,v,w$, with $a$ opposite $u$, $b$ opposite $v$, and $c$ opposite $w$.  In \cite[Proposition 7.1]{BS}, Bestvina and Savin demonstrate that $a+u = b+v = c+w$.  This excludes Form (IV), since the sum of two positive numbers cannot equal the sum of two negative numbers.  \cite[\S 7]{BS} also gives a formula for the discriminant,
$$\Delta = z^2 - 2au - 2bv - 2cw, \text{ where } z = a+u = b+v = c+w.$$
In forms (II) and (III), we may place $a,u,b,v$ so that $a$ and $u$ have opposite signs, and $b$ and $v$ have opposite signs.  Expressing $z$ as $c+w$ yields
$$\Delta = c^2 + w^2 - 2au - 2bv.$$
As the right side is a sum of positive terms, we find
$$\min \{ \vert a \vert, \vert b \vert, \vert c \vert, \vert u \vert, \vert v \vert, \vert w \vert \} \leq \sqrt{\Delta / 6}.$$
\end{proof}

\section{Real quadratic arithmetic}

In \cite[\S 4]{JW}, Johnson and Weiss give an explicit realization of the Coxeter groups of types $(4, \infty)$ and $(6, \infty)$ as arithmetic groups.  We describe this briefly here.  Let $\sigma = 2$ or $\sigma = 3$, and $R_\sigma = \ZZ[\sqrt{\sigma}]$.  The {\em dilinear group} (our own name) $DL_2(R_\sigma)$ is the group of all matrices $\Matrix{a}{b}{c}{d} \in GL_2(R_\sigma)$ such that
$$(a,d\in \ZZ \cdot \sqrt{\sigma} \text{ and } b,c \in {\ZZ}) \text{ or } (a,d \in \ZZ \text{ and } b,c \in {\ZZ} \cdot \sqrt{\sigma}).$$

Let $DL_2^+(R_\sigma)$ denote its subgroup consisting of matrices with $a,d \in \ZZ$ and $b,c \in \ZZ \cdot \sqrt{\sigma}$.  While $DL_2(R_\sigma)$ is a bit mysterious, $DL_2^+(R_\sigma)$ is $GL_2(\QQ(\sqrt{\sigma}))$-conjugate to a congruence subgroup of $GL_2(\ZZ)$:  if $g = \mathrm{diag}(1,\sqrt{\sigma})$, then 
$$g DL_2(R_\sigma) g^{-1} = \Gamma_0(\sigma) := \left\{ \Matrix{\alpha}{\beta}{\gamma}{\delta} \in GL_2(\ZZ) : \gamma \in \sigma \ZZ \right\}.$$
(We thank an anonymous referee for this insight!)

Define $PDL_2(R_\sigma)=DL_2(R_\sigma)/\{\pm \One \}$.  Johnson and Weiss present $PDL_2(R_\sigma)$ by generators and relations, giving an isomorphism $PDL_2(R_\sigma) \isom (2 \sigma, \infty)$.  Thus we expect arithmetic interpretations of the geometries of types $(4,\infty)$ and $(6, \infty)$.

\subsection{Arithmetic flags} 

We define a ``dilinear'' variant of Conways's topograph as follows.  As always, $\sigma = 2$ or $\sigma = 3$.
\begin{itemize}
	\item
	Faces correspond to {\em primitive lax divectors}: ordered pairs $(u, v\sqrt{\sigma})$ with $u,v \in \ZZ$ and $\GCD(u, \sigma v) = 1$ are called {\em primitive red divectors}.  Ordered pairs $(u \sqrt{\sigma},v)$ with $u,v \in \ZZ$ and $\GCD(\sigma u, v) = 1$ are called {\em primitive blue divectors}.  For laxness, we consider divectors modulo sign.
	\item
	Edges correspond to {\em lax dibases}:  unordered pairs of lax divectors generating $R_\sigma^2$ as an $R_\sigma$-module.  This implies that the divectors have opposite color, and form the rows of a matrix in $DL_2(R_\sigma)$.
	\item
	Points correspond to {\em lax pinwheels}: cyclically ordered $2\sigma$-tuples of lax divectors such that any adjacent pair forms a lax dibasis (and hence has opposite color).
\end{itemize}

\begin{thm}
The geometry of primitive lax divectors, lax dibases, and pinwheels for $R_\sigma$ is equivariantly isomorphic to the Coxeter geometry of type $(2 \sigma,\infty)$.
\end{thm}

\begin{figure}[htbp]
\begin{center}
\includegraphics[width=\linewidth]{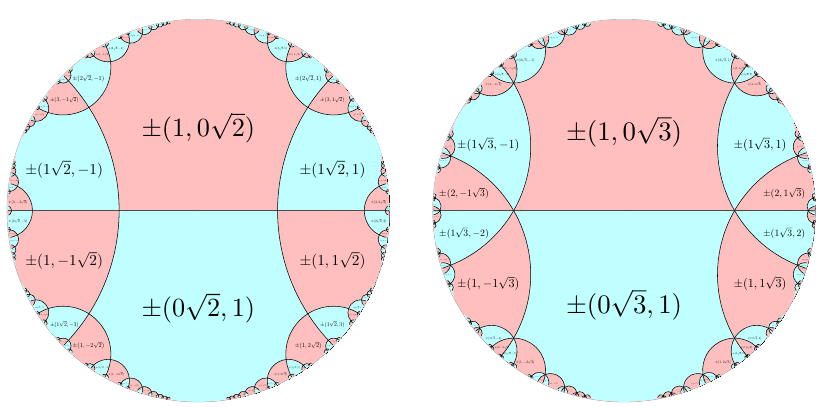}
%
\end{center}
\caption{The Coxeter geometries of type $(4,\infty)$ and $(6,\infty)$ are labeled by primitive lax divectors for $\ZZ[\sqrt{2}]$ and $\ZZ[\sqrt{3}]$, respectively.  Around each point is a pinwheel.}
\label{Pinwheels}
\end{figure}

\subsection{Binary quadratic diforms} 

Fix $\sigma = 2$ or $\sigma = 3$.  A binary quadratic diform (BQD) is a function of the form
$$Q(x,y)=ax^2+b\sqrt{\sigma} xy +cy^2, \text{ where } a,b,c\in \ZZ.$$
We restrict $(x,y)$ to be a divector in $R_\sigma^2$, so the values of $Q$ are integers.  We define the discriminant of $Q$ by $\Delta =\sigma(b^2 \sigma -4ac)$.

Restricting $Q$ to red and blue divectors yields a pair $Q_{\red}, Q_{\blue}$ of BQFs over $\ZZ$ of discriminant $\Delta$; explicitly,
$$Q_{\red}(u, v) := Q(u, v \sqrt{\sigma}) = a u^2 + b \sigma uv + c \sigma v^2,$$
$$Q_{\blue}(u,v) := Q(u \sqrt{\sigma}, v) = a \sigma u^2 + b \sigma uv + c v^2.$$
Define another BQF of discriminant $\Delta$,
$$A_\Delta(u,v) = \begin{cases} 
\sigma u^2 - \frac{\Delta}{4 \sigma} v^2 & \text{ if } \Delta \sigma^{-1} \ident 0 \text{ mod } 4; \\
\sigma u^2 + \sigma u v - \frac{\Delta - \sigma^2}{4 \sigma}v^2  & \text{ if } \Delta \sigma^{-1} \not \ident 0 \text{ mod } 4.
\end{cases}$$

Write $\Cl(\Delta)$ for the group of $SL_2(\ZZ)$-equivalence classes of primitive BQFs of discriminant $\Delta$, following Bhargava \cite[Theorem 1]{Bhargava}.  Then $[A_\Delta]$ is the unique such class which represents $\sigma$; since $A_\Delta$ is ambiguous, $[A_\Delta]^2 = 1$ in the class group.
\begin{thm}
\label{BQFBQD}
Suppose that $a$, $b \sigma$, and $c$ are pairwise coprime.  In $\Cl(\Delta)$, one has $[Q_{\red}] = [A_\Delta] \cdot [Q_{\blue}]$.  Conversely, if $Q_1, Q_2$ are primitive BQFs of discriminant $\Delta$, and $\sigma \mid \Delta$, and $[Q_1] = [A_\Delta] \cdot [Q_2]$, there exists a BQD $Q$ such that $[Q_{\red}] = [Q_1]$ and $[Q_{\blue}] = [Q_2]$.
\end{thm}
\begin{proof}
Dirichlet composition demonstrates that $[Q_{\red}] \cdot [Q_{\blue}]^{-1}$ is a form of discriminant $\Delta$ that represents $\sigma$.  Therefore $[Q_{\red}] \cdot [Q_{\blue}]^{-1} = [A_\Delta]$ in $\Cl(\Delta)$.
For the converse, if $Q_1$ is any BQF of discriminant $\Delta$, then a bit of algebra suffices to cook up a BQD $Q$ with $[Q_{\red}] = [Q_1]$.  If $[Q_1] = [A_\Delta] \cdot [Q_2]$, then the identiity $[Q_{\red}] = [A_\Delta] \cdot [Q_{\blue}]$ implies that $[Q_{\blue}] = [Q_2]$.

\end{proof}

The {\em topograph} of a BQD $Q$ is obtained by replacing each primitive lax divector by the corresponding value of $Q$.  As before, the discriminant is locally visible.  With $u,v,e,f,m,n$ as in Figure \ref{ArithProgRule}, we have
\begin{equation}
\label{locDelta}
\Delta = \begin{cases} 
(2u-v)^2-ef & \text{ if } \sigma = 2; \\
(3u-v)^2-ef=\frac{(4u-3v)^2-mn}{4} & \text{ if } \sigma = 3.
\end{cases}
\end{equation}  

\begin{figure}[htbp!]
		\centering
		\includegraphics[width=0.8\linewidth]{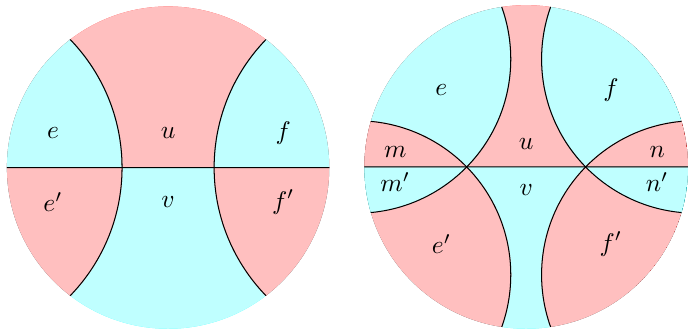}
		\caption{Cells in the range topograph for $\sigma=2$ (left) and $\sigma=3$ (right).}
		\label{ArithProgRule}\end{figure}

Polarization for the quadratic form $Q$ implies the following. 
\begin{thm}
	At every cell in the topograph of $Q$, as in Figure \ref{ArithProgRule}, one finds arithmetic progressions as below.
	\begin{description}
	\item[$\sigma = 2$] The triples $(e, 2u+v, f)$ and $(e', u+2v, f')$ are arithmetic progressions of the same step size. 
	\item[$\sigma = 3$] The triples $(e,3u+v,f)$ and $(e', u + 3v, f')$ are arithmetic progressions of the same step size $\delta$ and the triples $(m,4u+3v,n)$ and $(m', 3u + 4v, n')$ are arithmetic progressions of the same step size $2 \delta$. 
	\end{description}
\end{thm}

\begin{figure}[tbhp]
\begin{center}
\includegraphics[width=\linewidth]{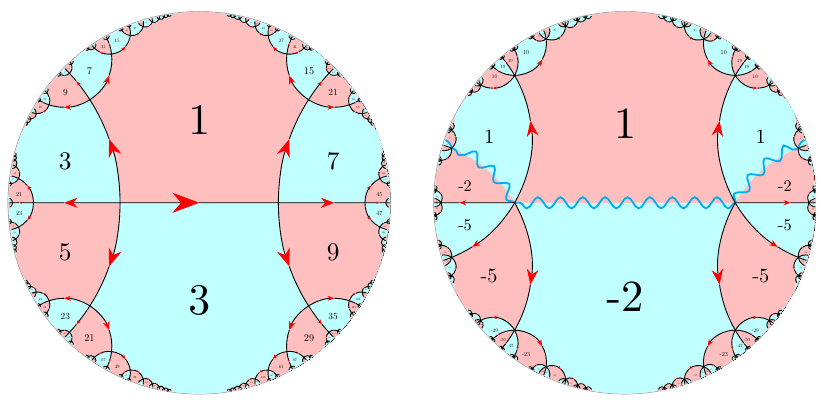}  
%
\end{center}
\caption{Topographs for the definite binary quadratic diform $Q(x,y) = x^2 + \sqrt{2} xy + 3y^2$ over $\ZZ[\sqrt{2}]$ and the indefinite diform $Q(x,y) = x^2 - 2y^2$ over $\ZZ[\sqrt{3}]$. }
\label{BQDTopos}
\end{figure}

Thus we draw an arrow on each edge to represent the direction of increasing progressions, or a circle if all progressions are constant. Figure \ref{BQDTopos} displays some examples. The climbing principle is the same as Conway's:  arrows always propagate when one looks at a cell of positive values.  Thus a positive-definite BQD has a unique source for its flow.  The binary quadratic forms $Q_{\red}$ and $Q_{\blue}$ may be {\em simultaneously} reduced.   

The {\em river} of a BQD is the set of segments separating positive values from negative.  The most interesting forms, just as for BQFs, are the nondegenerate indefinite forms.  The topographs of such forms exhibit a single endless and periodic river, as in Conway's case.  An analysis of riverbends gives a minimum value bound for such BQDs.

\begin{figure}[htbp!]
	\begin{center}
		\includegraphics[width=0.75\linewidth]{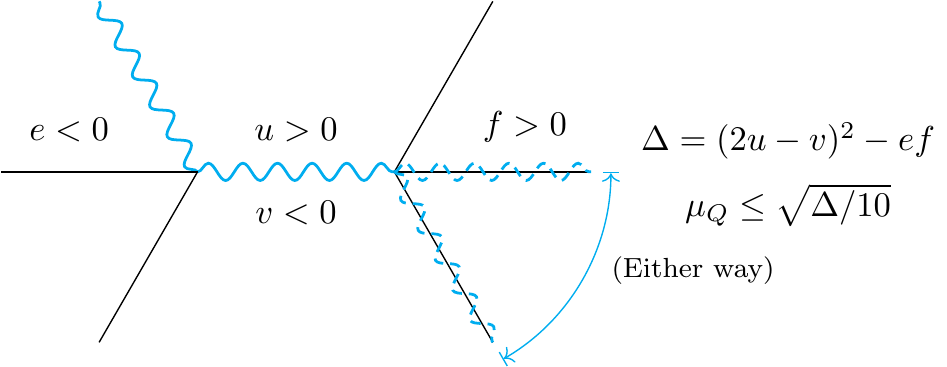} 
	\end{center}
	\caption{Riverbend types for $\sigma=2$.}
	\label{Riverbend_2}
\end{figure}

\begin{figure}[htbp]
	\begin{center}
		\includegraphics[width=0.75\linewidth]{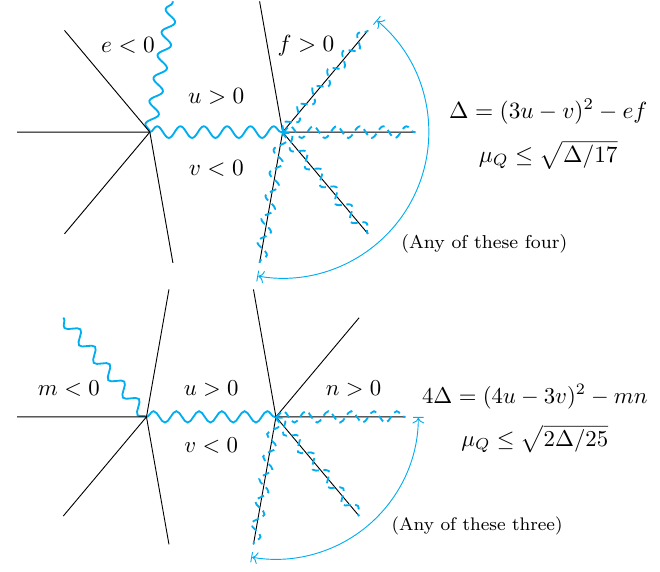}
	\end{center}
	\caption{Riverbend types for $\sigma=3$.}
	\label{Riverbend_3}
\end{figure}

\begin{thm}
	Let $Q$ be a nondegenerate indefinite BQD, and let $\mu_Q$ denote its minimum nonzero value.
	\begin{description}
		\item[$\sigma=2$:] If Q is not $DL_2(R_\sigma)$-equivalent to a multiple of $x^2 - y^2$, then $\vert \mu_Q \vert \leq \sqrt{\Delta / 10}$.
		\item[$\sigma=3$:] If Q is not $DL_2(R_\sigma)$-equivalent to a multiple of $x^2 - y^2$ or $x^2 - 2y^2$, then $\vert \mu_Q \vert \leq \sqrt{2 \Delta / 25}$.
	\end{description}	 
\end{thm}
\begin{proof}
The entire river cannot be adjacent to a single region, because its values opposite such a region would form a bi-infinite quadratic sequence with positive sign and negative acceleration, or negative sign and positive acceleration.  Hence the river must ``bend.''  If one finds riverbends as in Figures \ref{Riverbend_2} or \ref{Riverbend_3}, \eqref{locDelta} gives the stated minimum value bound or better (as derived in the figures).  If no such riverbends of those shapes occur, then the river must maintain one of the three shapes of Figure \ref{riverstraight} throughout its entire length. 

The isometry group of such a homogeneous river includes a translation along the river; a bit of matrix algebra implies that $Q$ is equivalent to a multiple of $x^2-y^2$ when $\sigma = 2$, or a multiple of $x^2 - y^2$ or $x^2 - 2y^2$ if $\sigma = 3$.

\begin{figure}
	\begin{center}
		\includegraphics[width=\linewidth]{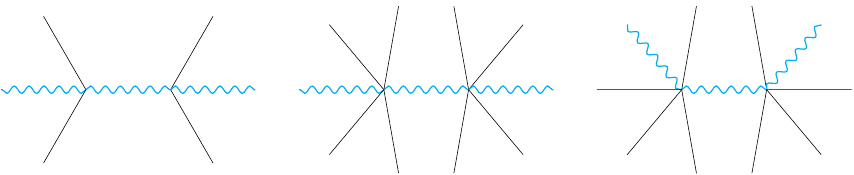}
	\end{center}
	\caption{One more river shape for $\sigma = 2$ and two more shapes for $\sigma = 3$.}
	\label{riverstraight}
\end{figure} 

\end{proof}

\begin{corollary}
Suppose that $Q_1$ and $Q_2$ are nondegenerate indefinite BQFs of discriminant $\Delta$, with $\sigma \mid \Delta$ and $[Q_2] = [A_\Delta] \cdot [Q_1]$.  Then
\begin{description}
\item[$\sigma = 2$:] If $Q_1$ and $Q_2$ are not equivalent to a multiple of $x^2 - 2y^2$, then $\vert \mu_{Q_1} \vert  \leq \sqrt{\Delta / 10}$ or $\vert \mu_{Q_2} \vert \leq \sqrt{\Delta / 10}.$
\item[$\sigma = 3$:] If $Q_1$ and $Q_2$ are not equivalent to a multiple of $x^2 - 3y^2$ or $2x^2 - 3y^2$, then $\vert \mu_{Q_1} \vert \leq \sqrt{\Delta / 13}$ or $\vert \mu_{Q_2} \vert \leq \sqrt{\Delta / 13}.$
\end{description}
\end{corollary}
\begin{proof}
This follows directly from the previous theorem and Theorem \ref{BQFBQD}, except that $2/25$ has been replaced by $1/13$.  This is possible, due to a gap in the Markoff spectrum between $\sqrt{12}$ and $\sqrt{13}$; see \cite[\S 1, Proof of Theorem 3.3]{Hall}.
\end{proof}

\section{Conclusion}

In each of the discussed examples, there is a coincidence between a Coxeter group and an arithmetic group.  For Conway's topograph, it is the coincidence between the Coxeter group of type $(3, \infty)$ and the arithmetic group $PGL_2(\ZZ)$.  The dilinear groups, of Coxeter types $(2 \sigma, \infty)$ for $\sigma=2,3$, are arithmetic subgroups of $PGU_{1,1}^{\QQ(\sqrt{\sigma}) / \QQ}(\QQ)$, the projective unitary similitude group for a Hermitian form relative to $\QQ(\sqrt{\sigma}) / \QQ$.  

When such a coincidence occurs, the Coxeter group is arithmetic, and the following two questions are natural.

\begin{enumerate}
\item
Is there an arithmetic interpretation for the flags in the Coxeter group?
\item
Does the Coxeter geometry give a new reduction theory for a class of quadratic (or Hermitian) forms?
\end{enumerate}

The first question is reminiscent of the classical theory of flag varieties.  When $\alg{G}$ is a simple simply-connected linear algebraic group over a field $k$, one can often identify a ``standard'' representation of $\alg{G}$ on a $k$-vector space $V$.  Every $k$-parabolic subgroup of $\alg{G}$ is the stabilizer of some sort of $k$-flag in $V$.  If $\alg{G}$ is a symplectic or spin or unitary group, these are the isotropic flags in the standard representation.  In type $G_2$, these are the nil-flags in the split octonions.  In an 11-part series of papers ({\em Beziehungen der $\mathfrak{E}7$ und $\mathfrak{E}8$ zur Oktavenebene I--XI}, published 1954--63, ending with \cite{Freudenthal}), Freudenthal studied the ``metasymplectic'' geometry which describes flags in representations of exceptional groups.

Now it appears that arithmetic Coxeter groups provide a parallel industry, examining their representations on various modules over Euclidean domains.  Arithmetic flags are generalized bases of these modules.  The geometry of arithmetic flag varieties seems (so far) to be the combinatorial geometry of Coxeter groups.  We do not yet see algebraic geometry in the picture, as one finds in flag varieties $\alg{G} / \alg{P}$.

The applications to arithmetic (the {\em arithmetic} of arithmetic Coxeter groups) include Conway's approach to binary quadratic forms and new generalizations.  The reduction theory for quadratic and Hermitian forms is classical subject, sometimes tedious in its algebra -- the Coxeter geometry and Conway's theory of wells and rivers gives an intuitive approach.  Beyond reframing old results, it seems unlikely that one would find the reduction theory of our ``diforms'' (or suitable pairs of binary quadratic forms) without considering the Coxeter group.  In this way, arithmetic Coxeter groups offer new applications to number theory.

This paper has discussed five arithmetic Coxeter groups, of types $(3, \infty)$, $(3,3,6)$, $(3,4,4)$, $(4, \infty)$, and $(6, \infty)$.  If this is a game of coincidences, when might it end?  In \cite{CoxSurvey}, Belolipetsky surveys the arithmetic {\em hyperbolic} Coxeter groups; following his treatment, we review the classification of such Coxeter groups.  

The groups we have studied are {\em simplicial hyperbolic} arithmetic Coxeter groups.  In \cite{Vinberg}, Vinberg proves there are 64 such groups in dimension at least 3.  These fall into 14 commensurability classes by \cite{JKRT}, as shown in Table \ref{SHAcox}.  It would not be surprising if each offered a notion of arithmetic flags (e.g., superbases, etc.) and quadratic/Hermitian forms.  For example, the Coxeter group of type $(3,3,3,4,3)$ is arithmetic, commensurable with $PGL_2(A)$ where $A$ is the Hurwitz order in the quaternion algebra $\QQ + \QQ i + \QQ j + \QQ k$.  Arithmetic flags in this case can be interpreted as lax vectors, bases, superbases, 3-simplex-bases, 4-simplex-bases, and 5-orthoplex-bases, in the $A$-module $A^2$.

\begin{table}[thbp!]
\centering
\caption{Commensurability classes of simplicial hyperbolic arithmetic Coxeter groups of dimension at least 3, (extracted from \cite{JKRT}).}
\label{SHAcox}
\begin{tabular}{lc}
Dimension & Coxeter Types \\
\midrule
3 & $(3,3,6)$ and $(3,4,4)$ \\
4 & $(3,3,3,5)$ and $(3,3,3,4)$ and $(3,4,3,4)$  \\
5 & $(3,3,3,4,3)$ and $(3,3^{[5]})$ \\
6 & $(4,3^2,3^{2,1})$ and $(3,3^{[6]})$ \\
7 & $(3^{2,2,2})$ and $(4,3^3,3^{2,1})$ and $(3,3^{[7]})$ \\
8 & $(3^{4,3,1})$ \\
9 & $(3^{6,2,1})$ \\
\bottomrule
\end{tabular}
\end{table}

Table \ref{SHAcox} only displays groups of dimension at least three.  In dimension two, we find Conway's topograph and its dilinear variants.  One might also consider arithmetic hyperbolic {\em triangle} groups, classified by Takeuchi in \cite{Takeuchi}, \cite{Takeuchi2}.  Up to commensurability, there are 19 of these, each associated to a quaternion algebra over a totally real field.  Vertices, edges, and triangles in the resulting hyperbolic tilings surely correspond to arithmetic objects -- what are they?  

If one wishes to depart from the simplicial groups, there are non-simplicial arithmetic hyperbolic Coxeter groups.  By results of Vinberg \cite{Vinberg81}, all examples occur in dimension at most 30; there are finitely many up to commensurability.  One may be able to explore the arithmetic of arithmetic Coxeter groups for a long time -- what is currently missing is a general theory of arithmetic flags and forms to make predictions in a less ad hoc manner. 

Departing the setting of Coxeter groups may also be appealing, especially in low dimension.  For example, the Coxeter geometry makes the reduction theory of binary Hermitian forms particularly nice over $\ZZ[i]$ and $\ZZ[\omega]$.  But Bestvina and Savin \cite{BS} are able to work over other quadratic imaginary rings although the geometry lacks homogeneity.  One might study diforms over other real quadratic rings, in the same way.  More arithmetic may be found in ``thin'' rather than arithmetic groups, e.g., in the work of Stange \cite{Stange} on Apollonian circle packings.  Still, Coxeter groups seem an appropriate starting place, where arithmetic applications are low-hanging fruit.

\section*{Acknowledgment}

M. Weissman is supported by the Simons Foundation Collaboration Grant \#426453.  The authors thank an anonymous referee for comments and insights.  The Eisenstein topograph was first described in 2007, in the unpublished Masters thesis of Andreas Weinert.

\bibliographystyle{amsalpha}
\bibliography{ArithHyp}

\providecommand{\bysame}{\leavevmode\hbox to3em{\hrulefill}\thinspace}
\providecommand{\MR}{\relax\ifhmode\unskip\space\fi MR }
\providecommand{\MRhref}[2]{%
  \href{http://www.ams.org/mathscinet-getitem?mr=#1}{#2}
}
\providecommand{\href}[2]{#2}
\begin{thebibliography}{JKRT02}

\bibitem[Bel16]{CoxSurvey}
Mikhail Belolipetsky, \emph{Arithmetic hyperbolic reflection groups}, Bull.
  Amer. Math. Soc. (N.S.) \textbf{53} (2016), no.~3, 437--475. \MR{3501796}

\bibitem[Bha04]{Bhargava}
Manjul Bhargava, \emph{Higher composition laws. {I}. {A} new view on {G}auss
  composition, and quadratic generalizations}, Ann. of Math. (2) \textbf{159}
  (2004), no.~1, 217--250. \MR{2051392}

\bibitem[Bia91]{Bianchi}
Luigi Bianchi, \emph{Geometrische {D}arstellung der {G}ruppen linearer
  {S}ubstitutionen mit ganzen complexen {C}oefficienten nebst {A}nwendungen auf
  die {Z}ahlentheorie}, Math. Ann. \textbf{38} (1891), no.~3, 313--333.
  \MR{1510677}

\bibitem[BS12]{BS}
Mladen Bestvina and Gordan Savin, \emph{Geometry of integral binary {H}ermitian
  forms}, J. Algebra \textbf{360} (2012), 1--20. \MR{2914631}

\bibitem[Con97]{Conway}
John~H. Conway, \emph{The sensual (quadratic) form}, Carus Mathematical
  Monographs, vol.~26, Mathematical Association of America, Washington, DC,
  1997, With the assistance of Francis Y. C. Fung. \MR{1478672}

\bibitem[Cox49]{Coxeter}
H.~S.~M. Coxeter, \emph{Regular {P}olytopes}, Methuen \& Co., Ltd., London;
  Pitman Publishing Corporation, New York, 1948; 1949. \MR{0027148}

\bibitem[FK97]{FK}
R.~Fricke and F.~Klein, \emph{Vorlesungen {\"u}ber die theorie der automorphen
  functionen}, vol.~1, Leipzig, B.G. Teubner, 1897.

\bibitem[Fre63]{Freudenthal}
Hans Freudenthal, \emph{Beziehungen der {${\mathfrak E}_{7}$} und {${\mathfrak
  E}_{8}$} zur {O}ktavenebene. {X}, {XI}}, Nederl. Akad. Wetensch. Proc. Ser. A
  66 = Indag. Math. \textbf{25} (1963), 457--471; 472--487. \MR{0163203}

\bibitem[Hal71]{Hall}
Marshall Hall, Jr., \emph{The {M}arkoff spectrum}, Acta Arith. \textbf{18}
  (1971), 387--399. \MR{0296023}

\bibitem[JKRT02]{JKRT}
N.~W. Johnson, R.~Kellerhals, J.~G. Ratcliffe, and S.~T. Tschantz,
  \emph{Commensurability classes of hyperbolic {C}oxeter groups}, Linear
  Algebra Appl. \textbf{345} (2002), 119--147. \MR{1883270}

\bibitem[JW99]{JW}
Norman~W. Johnson and Asia~Ivi\'c Weiss, \emph{Quadratic integers and {C}oxeter
  groups}, Canad. J. Math. \textbf{51} (1999), no.~6, 1307--1336, Dedicated to
  H. S. M. Coxeter on the occasion of his 90th birthday. \MR{1756885}

\bibitem[Sta16]{Stange}
Katherine~E. Stange, \emph{The sensual {A}pollonian circle packing}, Expo.
  Math. \textbf{34} (2016), no.~4, 364--395. \MR{3578004}

\bibitem[SW94]{SW}
Egon Schulte and Asia~Ivi\'c Weiss, \emph{Chirality and projective linear
  groups}, Discrete Math. \textbf{131} (1994), no.~1-3, 221--261. \MR{1287736}

\bibitem[Tak77a]{Takeuchi2}
Kisao Takeuchi, \emph{Arithmetic triangle groups}, J. Math. Soc. Japan
  \textbf{29} (1977), no.~1, 91--106. \MR{0429744}

\bibitem[Tak77b]{Takeuchi}
\bysame, \emph{Commensurability classes of arithmetic triangle groups}, J. Fac.
  Sci. Univ. Tokyo Sect. IA Math. \textbf{24} (1977), no.~1, 201--212.
  \MR{0463116}

\bibitem[Vin67]{Vinberg}
\`E.~B. Vinberg, \emph{Discrete groups generated by reflections in {L}oba\v
  cevski\u\i \ spaces}, Mat. Sb. (N.S.) \textbf{72 (114)} (1967), 471--488;
  correction, ibid. 73 (115) (1967), 303. \MR{0207853}

\bibitem[Vin81]{Vinberg81}
\bysame, \emph{The nonexistence of crystallographic reflection groups in
  {L}obachevski\u\i \ spaces of large dimension}, Funktsional. Anal. i
  Prilozhen. \textbf{15} (1981), no.~2, 67--68. \MR{617472}

\end{thebibliography}

\end{document}